\documentclass[a4paper,10pt]{amsart}

\usepackage[utf8]{inputenc}
\usepackage{url}

\usepackage[expansion=false]{microtype}
\usepackage{amssymb,amsmath,amsopn,amsxtra,amsthm,amsfonts}
\usepackage{xr}
\usepackage[mathcal]{euscript}
\usepackage{mathalfa}
\usepackage{txfonts}
\usepackage{todonotes}
\usepackage[all,cmtip]{xy}
\usepackage{xfrac}

\DeclareMathAlphabet{\mathbb}{U}{msb}{m}{n}

\usepackage[pdfusetitle,unicode]{hyperref}

\usepackage{cite}
\usepackage{hyperref}
\newtheorem{cor}[subsection]{Corollary}
\newtheorem{lem}[subsection]{Lemma}
\newtheorem{prop}[subsection]{Proposition}

\newtheorem{thm}{Theorem}

\newtheorem*{thm*}{Theorem}

\newtheorem*{conj*}{Conjecture}

\newtheorem*{quest*}{Problem}

\numberwithin{equation}{thm}
\theoremstyle{definition}

\newtheorem{remark}[subsection]{Remark}

\theoremstyle{remark}

\renewcommand{\eqref}[1]{(\ref{#1})}

\usepackage{tikz}
\usetikzlibrary{matrix}
\usepackage{tikz-cd}
\tikzset{shorten <>/.style={shorten >=#1,shorten <=#1}}
\tikzcdset{arrow style=tikz,diagrams={>={Straight Barb[scale=0.8]}}}

\usepackage{fixltx2e}

\usepackage{xspace}

\usepackage{xifthen}

\usepackage{xparse}

\newcommand{\nc}{\newcommand}
\nc{\renc}{\renewcommand}
\nc{\ssec}{\subsection}
\nc{\sssec}{\subsubsection}
\nc{\on}{\operatorname}
\nc{\term}[1]{#1\xspace}

\DeclareMathSymbol{A}{\mathalpha}{operators}{`A}
\DeclareMathSymbol{B}{\mathalpha}{operators}{`B}
\DeclareMathSymbol{C}{\mathalpha}{operators}{`C}
\DeclareMathSymbol{D}{\mathalpha}{operators}{`D}
\DeclareMathSymbol{E}{\mathalpha}{operators}{`E}
\DeclareMathSymbol{F}{\mathalpha}{operators}{`F}
\DeclareMathSymbol{G}{\mathalpha}{operators}{`G}
\DeclareMathSymbol{H}{\mathalpha}{operators}{`H}
\DeclareMathSymbol{I}{\mathalpha}{operators}{`I}
\DeclareMathSymbol{J}{\mathalpha}{operators}{`J}
\DeclareMathSymbol{K}{\mathalpha}{operators}{`K}
\DeclareMathSymbol{L}{\mathalpha}{operators}{`L}
\DeclareMathSymbol{M}{\mathalpha}{operators}{`M}
\DeclareMathSymbol{N}{\mathalpha}{operators}{`N}
\DeclareMathSymbol{O}{\mathalpha}{operators}{`O}
\DeclareMathSymbol{P}{\mathalpha}{operators}{`P}
\DeclareMathSymbol{Q}{\mathalpha}{operators}{`Q}
\DeclareMathSymbol{R}{\mathalpha}{operators}{`R}
\DeclareMathSymbol{S}{\mathalpha}{operators}{`S}
\DeclareMathSymbol{T}{\mathalpha}{operators}{`T}
\DeclareMathSymbol{U}{\mathalpha}{operators}{`U}
\DeclareMathSymbol{V}{\mathalpha}{operators}{`V}
\DeclareMathSymbol{W}{\mathalpha}{operators}{`W}
\DeclareMathSymbol{X}{\mathalpha}{operators}{`X}
\DeclareMathSymbol{Y}{\mathalpha}{operators}{`Y}
\DeclareMathSymbol{Z}{\mathalpha}{operators}{`Z}

\nc{\sA}{\ensuremath{\mathcal{A}}\xspace}
\nc{\sB}{\ensuremath{\mathcal{B}}\xspace}
\nc{\sC}{\ensuremath{\mathcal{C}}\xspace}
\nc{\sD}{\ensuremath{\mathcal{D}}\xspace}
\nc{\sE}{\ensuremath{\mathcal{E}}\xspace}
\nc{\sF}{\ensuremath{\mathcal{F}}\xspace}
\nc{\sG}{\ensuremath{\mathcal{G}}\xspace}
\nc{\sH}{\ensuremath{\mathcal{H}}\xspace}
\nc{\sI}{\ensuremath{\mathcal{I}}\xspace}
\nc{\sJ}{\ensuremath{\mathcal{J}}\xspace}
\nc{\sK}{\ensuremath{\mathcal{K}}\xspace}
\nc{\sL}{\ensuremath{\mathcal{L}}\xspace}
\nc{\sM}{\ensuremath{\mathcal{M}}\xspace}
\nc{\sN}{\ensuremath{\mathcal{N}}\xspace}
\nc{\sO}{\ensuremath{\mathcal{O}}\xspace}
\nc{\sP}{\ensuremath{\mathcal{P}}\xspace}
\nc{\sQ}{\ensuremath{\mathcal{Q}}\xspace}
\nc{\sR}{\ensuremath{\mathcal{R}}\xspace}
\nc{\sS}{\ensuremath{\mathcal{S}}\xspace}
\nc{\sT}{\ensuremath{\mathcal{T}}\xspace}
\nc{\sU}{\ensuremath{\mathcal{U}}\xspace}
\nc{\sV}{\ensuremath{\mathcal{V}}\xspace}
\nc{\sW}{\ensuremath{\mathcal{W}}\xspace}
\nc{\sX}{\ensuremath{\mathcal{X}}\xspace}
\nc{\sY}{\ensuremath{\mathcal{Y}}\xspace}
\nc{\sZ}{\ensuremath{\mathcal{Z}}\xspace}

\nc{\bA}{\ensuremath{\mathbf{A}}\xspace}
\nc{\bB}{\ensuremath{\mathbf{B}}\xspace}
\nc{\bC}{\ensuremath{\mathbf{C}}\xspace}
\nc{\bD}{\ensuremath{\mathbf{D}}\xspace}
\nc{\bE}{\ensuremath{\mathbf{E}}\xspace}
\nc{\bF}{\ensuremath{\mathbf{F}}\xspace}
\nc{\bG}{\ensuremath{\mathbf{G}}\xspace}
\nc{\bH}{\ensuremath{\mathbf{H}}\xspace}
\nc{\bI}{\ensuremath{\mathbf{I}}\xspace}
\nc{\bJ}{\ensuremath{\mathbf{J}}\xspace}
\nc{\bK}{\ensuremath{\mathbf{K}}\xspace}
\nc{\bL}{\ensuremath{\mathbf{L}}\xspace}
\nc{\bM}{\ensuremath{\mathbf{M}}\xspace}
\nc{\bN}{\ensuremath{\mathbf{N}}\xspace}
\nc{\bO}{\ensuremath{\mathbf{O}}\xspace}
\nc{\bP}{\ensuremath{\mathbf{P}}\xspace}
\nc{\bQ}{\ensuremath{\mathbf{Q}}\xspace}
\nc{\bR}{\ensuremath{\mathbf{R}}\xspace}
\nc{\bS}{\ensuremath{\mathbf{S}}\xspace}
\nc{\bT}{\ensuremath{\mathbf{T}}\xspace}
\nc{\bU}{\ensuremath{\mathbf{U}}\xspace}
\nc{\bV}{\ensuremath{\mathbf{V}}\xspace}
\nc{\bW}{\ensuremath{\mathbf{W}}\xspace}
\nc{\bX}{\ensuremath{\mathbf{X}}\xspace}
\nc{\bY}{\ensuremath{\mathbf{Y}}\xspace}
\nc{\bZ}{\ensuremath{\mathbf{Z}}\xspace}

\nc{\dA}{\ensuremath{\mathds{A}}\xspace}
\nc{\dB}{\ensuremath{\mathds{B}}\xspace}
\nc{\dC}{\ensuremath{\mathds{C}}\xspace}
\nc{\dD}{\ensuremath{\mathds{D}}\xspace}
\nc{\dE}{\ensuremath{\mathds{E}}\xspace}
\nc{\dF}{\ensuremath{\mathds{F}}\xspace}
\nc{\dG}{\ensuremath{\mathds{G}}\xspace}
\nc{\dH}{\ensuremath{\mathds{H}}\xspace}
\nc{\dI}{\ensuremath{\mathds{I}}\xspace}
\nc{\dJ}{\ensuremath{\mathds{J}}\xspace}
\nc{\dK}{\ensuremath{\mathds{K}}\xspace}
\nc{\dL}{\ensuremath{\mathds{L}}\xspace}
\nc{\dM}{\ensuremath{\mathds{M}}\xspace}
\nc{\dN}{\ensuremath{\mathds{N}}\xspace}
\nc{\dO}{\ensuremath{\mathds{O}}\xspace}
\nc{\dP}{\ensuremath{\mathds{P}}\xspace}
\nc{\dQ}{\ensuremath{\mathds{Q}}\xspace}
\nc{\dR}{\ensuremath{\mathds{R}}\xspace}
\nc{\dS}{\ensuremath{\mathds{S}}\xspace}
\nc{\dT}{\ensuremath{\mathds{T}}\xspace}
\nc{\dU}{\ensuremath{\mathds{U}}\xspace}
\nc{\dV}{\ensuremath{\mathds{V}}\xspace}
\nc{\dW}{\ensuremath{\mathds{W}}\xspace}
\nc{\dX}{\ensuremath{\mathds{X}}\xspace}
\nc{\dY}{\ensuremath{\mathds{Y}}\xspace}
\nc{\dZ}{\ensuremath{\mathds{Z}}\xspace}

\nc{\bbA}{\ensuremath{\mathbb{A}}\xspace}
\nc{\bbB}{\ensuremath{\mathbb{B}}\xspace}
\nc{\bbC}{\ensuremath{\mathbb{C}}\xspace}
\nc{\bbD}{\ensuremath{\mathbb{D}}\xspace}
\nc{\bbE}{\ensuremath{\mathbb{E}}\xspace}
\nc{\bbF}{\ensuremath{\mathbb{F}}\xspace}
\nc{\bbG}{\ensuremath{\mathbb{G}}\xspace}
\nc{\bbH}{\ensuremath{\mathbb{H}}\xspace}
\nc{\bbI}{\ensuremath{\mathbb{I}}\xspace}
\nc{\bbJ}{\ensuremath{\mathbb{J}}\xspace}
\nc{\bbK}{\ensuremath{\mathbb{K}}\xspace}
\nc{\bbL}{\ensuremath{\mathbb{L}}\xspace}
\nc{\bbM}{\ensuremath{\mathbb{M}}\xspace}
\nc{\bbN}{\ensuremath{\mathbb{N}}\xspace}
\nc{\bbO}{\ensuremath{\mathbb{O}}\xspace}
\nc{\bbP}{\ensuremath{\mathbb{P}}\xspace}
\nc{\bbQ}{\ensuremath{\mathbb{Q}}\xspace}
\nc{\bbR}{\ensuremath{\mathbb{R}}\xspace}
\nc{\bbS}{\ensuremath{\mathbb{S}}\xspace}
\nc{\bbT}{\ensuremath{\mathbb{T}}\xspace}
\nc{\bbU}{\ensuremath{\mathbb{U}}\xspace}
\nc{\bbV}{\ensuremath{\mathbb{V}}\xspace}
\nc{\bbW}{\ensuremath{\mathbb{W}}\xspace}
\nc{\bbX}{\ensuremath{\mathbb{X}}\xspace}
\nc{\bbY}{\ensuremath{\mathbb{Y}}\xspace}
\nc{\bbZ}{\ensuremath{\mathbb{Z}}\xspace}

\nc{\mrm}[1]{\ensuremath{\mathrm{#1}}\xspace}
\nc{\mbf}[1]{\ensuremath{\mathbf{#1}}\xspace}
\nc{\mcal}[1]{\ensuremath{\mathcal{#1}}\xspace}
\nc{\msc}[1]{\ensuremath{\mathscr{#1}}\xspace}

\renc{\bar}[1]{\overline{#1}}

\nc{\sub}{\subset}
\nc{\too}{\longrightarrow}
\nc{\hook}{\hookrightarrow}
\nc*{\hooklongrightarrow}{\ensuremath{\lhook\joinrel\relbar\joinrel\rightarrow}}
\nc{\hooklong}{\hooklongrightarrow}
\nc{\twoheadlongrightarrow}{\relbar\joinrel\twoheadrightarrow}
\nc{\shiso}{\approx}
\nc{\isoto}{\xrightarrow{\sim}}
\nc{\isofrom}{\xleftarrow{\sim}}
\renc{\ge}{\geqslant}
\renc{\le}{\leqslant}
\renc{\geq}{\geqslant}
\renc{\leq}{\leqslant}

\nc{\id}{\mathrm{id}}
\nc{\can}{\mathrm{can}}

\DeclareMathOperator{\Hom}{\mathrm{Hom}}
\nc{\uHom}{\underline{\smash{\Hom}}}

\DeclareMathOperator{\End}{\mathrm{End}}

\DeclareMathOperator{\Map}{\mathrm{Map}}
\nc{\Pre}{\mathrm{PSh}{}}
\nc{\uEnd}{\underline{\smash{\End}}}

\renc{\lim}{\operatorname*{lim}}
\nc{\colim}{\operatorname*{colim}}
\nc{\Cofib}{\on{Cofib}}
\nc{\Fib}{\on{Fib}}
\nc{\initial}{\varnothing}
\nc{\op}{\mathrm{op}}

\nc{\bDelta}{\mbf{\Delta}}
\nc{\DM}{\mbf{DM}}
\nc{\eff}{\mathrm{eff}}
\nc{\veff}{\mathrm{veff}}
\nc{\hzmw}{\mathrm{H}\tilde{\Z}}
\nc{\hz}{\mathrm{H}\Z}
\nc{\cyc}{{\mrm{cyc}}}
\nc{\betah}{{\tilde{\beta}}}
\nc{\corr}{{\on{corr}}}
\nc{\fet}{{\mrm{f\acute et}}}
\nc{\fsyn}{{\mrm{fsyn}}}
\nc{\fflat}{{\mrm{fflat}}}
\nc{\syn}{{\mrm{syn}}}
\nc{\Gor}{{\mrm{Gor}}}
\nc{\MU}{{\mrm{MU}}}
\nc{\W}{{\mrm{W}}}
\nc{\gor}{{\mrm{gor}}}
\nc{\Pic}{\mrm{Pic}}
\nc{\perf}{\mrm{perf}}
\nc{\oblv}{\on{oblv}}
\nc{\exact}{\on{exact}}

\nc{\F}{{\mathcal{F}}}
\nc{\Aa}{{\mathcal{A}}}
\nc{\FF}{{\mathbb{F}}}
\nc{\clopen}{{\mrm{clopen}}}
\nc{\B}{\mrm{B}}
\nc{\D}{\mrm{D}}
\nc{\Ss}{\mathbb{S}}
\nc{\Fin}{\on{Fin}}
\nc{\Cut}{\on{Cut}}
\nc{\Cart}{\on{Cart}}
\nc{\pairs}{\mathsf{pairs}}
\nc{\Pairs}{\mathrm{Pair}}
\nc{\Trip}{\mathrm{Trip}}
\nc{\Lab}{\mathrm{Lab}}
\nc{\coCart}{\mathrm{coCart}}
\nc{\RKE}{\mathrm{RKE}}
\nc{\strict}{\mathrm{strict}}
\nc{\Emb}{\mathrm{Emb}}
\nc{\EMB}{\mathcal{E}\mathrm{mb}}
\nc{\Split}{\mathrm{Split}}
\nc{\Set}{\mathrm{Set}}
\nc{\sSets}{\mathrm{sSets}}
\nc{\pb}{\mathrm{pb}}
\nc{\lci}{\mathrm{lci}}
\nc{\fib}{\mathrm{fib}}
\nc{\cofib}{\mathrm{cofib}}
\nc{\diff}{\mrm{diff}}
\nc{\gp}{\mrm{gp}}
\nc{\chr}{\mrm{char}}
\nc{\mgp}{\mrm{mot-gp}}
\nc{\FSyn}{\mrm{FSyn}}
\nc{\FFlat}{{\mrm{FFlat}}}
\nc{\FEt}{\mrm{FEt}}
\nc{\Spc}{\mrm{Spc}}
\nc{\Ob}{\mrm{Ob}}
\nc{\Spt}{\mrm{Spt}}
\nc{\T}{\bT}
\nc{\suspinf}{\Sigma^\infty}
\nc{\h}{\mrm{h}}
\nc{\uhom}{\underline{\mathrm{Hom}}}
\nc{\umap}{\underline{\mathrm{Maps}}}
\renc{\H}{\bH}
\nc{\Einfty}{{\mathbb{E}_\infty}}
\nc{\Eone}{{\sE_1}}
\nc{\Stab}{\mrm{Stab}}
\nc{\lax}{{\mrm{lax}}}
\nc{\cocart}{{\mrm{cocart}}}
\nc{\Sch}{\mrm{Sch}}
\nc{\CH}{\mrm{CH}}
\nc{\dSch}{\mrm{dSch}}
\nc{\Aff}{\mrm{Aff}}
\nc{\SmAff}{\mrm{SmAff}}
\nc{\dAff}{\mrm{dAff}}
\nc{\Fr}{\on{Fr}}
\nc{\A}{\mathbf{A}}
\nc{\N}{\mathbf{N}}
\nc{\Z}{\mathbb{Z}}
\nc{\Q}{\mathbf{Q}}
\nc{\GW}{\mathbf{GW}}
\nc{\GWspace}{\mathcal{GW}}
\nc{\KGW}{\mathbb G\mathrm W}
\nc{\GWspectrum}{\mathrm{GW}}
\nc{\uGW}{\underline{\mathrm{GW}}}
\nc{\uW}{\underline{\mathrm{W}}}
\nc{\uZ}{\underline{\mathbf{Z}}}
\nc{\Oo}{\mathcal{O}} 
\nc{\red}{{\on{red}}}
\nc{\Voev}{{\on{Voev}}}
\nc{\Corr}{\mrm{Corr}}
\nc{\Span}{\mathbf{Corr}}
\nc{\Gap}{\mrm{Gap}}
\nc{\Filt}{\mrm{Filt}}
\nc{\Corrfr}{\Corr^{\fr}}
\nc{\Corrvfr}{\Corr^{\Vfr}}
\nc{\Spec}{\on{Spec}}
\nc{\Sm}{\mrm{Sm}}
\nc{\QSm}{\mrm{QSm}}
\nc{\Gm}{\mathbf{G}_{\mrm{m}}}
\renc{\P}{\bP}
\nc{\Nis}{\mathrm{Nis}}
\nc{\Zar}{\mathrm{Zar}}
\nc{\et}{\mathrm{\acute et}}
\nc{\all}{\mathrm{all}}
\nc{\fold}{\mathrm{fold}}
\nc{\Fun}{\mathrm{Fun}}
\nc{\Ho}{\mathrm{Ho}}
\nc{\Segal}{\mathrm{Segal}}
\nc{\Mon}{\mrm{Mon}{}}
\nc{\Ab}{\mrm{Ab}}
\nc{\Gr}{\mrm{Gr}}
\nc{\GrO}{\mrm{GrO}}
\nc{\Sh}{\on{Sh}}
\nc{\M}{\mrm{M}}
\nc{\Lhtp}{L_{\A^1}}
\nc{\Lmot}{L_{\mrm{mot}}}
\nc{\mot}{\mrm{mot}}
\nc{\SH}{\mbf{SH}}
\nc{\RR}{\mbf{R}}
\nc{\CC}{\mbf{C}}
\nc{\Mod}{\mrm{Mod}}
\nc{\QCoh}{\mrm{QCoh}}
\nc{\PSh}{\mrm{PSh}}
\nc{\MonUnit}{\mbf{1}}
\nc{\tr}{\on{tr}}
\nc{\vop}{\mrm{vop}}
\nc{\fr}{{\on{fr}}}
\nc{\Ar}{\mrm{Ar}}
\nc{\Vfr}{\on{Vfr}}
\nc{\frdiff}{{\on{frdiff}}}
\nc{\frGys}{\on{frGys}}
\nc{\SHfr}{\SH^{\fr}}
\nc{\SHfrdiff}{\SH^{\frdiff}}
\nc{\SHfrGys}{\SH^{\frGys}}
\nc{\InftyCat}{\infty\textnormal{-}\mrm{Cat}}
\nc{\TriCat}{\mathrm{TriCat}}
\nc{\Cat}{\mathrm{1\textnormal{-}Cat}}
\nc{\Th}{\on{Th}}

\nc{\CMon}{\mrm{CMon}{}}
\nc{\CAlg}{\mrm{CAlg}{}}
\nc{\MGL}{\mrm{MGL}}
\nc{\PMGL}{\mathrm{PMGL}}
\nc{\KGL}{\mrm{KGL}}
\nc{\kgl}{\mrm{kgl}}
\nc{\KQ}{\mrm{KQ}}
\nc{\kq}{\mrm{kq}}
\nc{\KSp}{\mrm{KSp}}
\nc{\ksp}{\mrm{ksp}}
\nc{\MSL}{\mrm{MSL}}
\nc{\MSp}{\mrm{MSp}}
\nc{\Seg}{\mrm{Seg}{}}
\nc{\Tw}{\mrm{Tw}}
\nc{\sslash}{/\mkern-6mu/}
\nc{\PrL}{\mrm{Pr}^\mrm{L}}
\nc{\PrR}{\mrm{Pr}^\mrm{R}}
\nc{\pr}{\mrm{pr}}
\nc{\efr}{\mrm{efr}}
\nc{\nfr}{\mrm{nfr}}
\nc{\dfr}{\mrm{fr}}
\nc{\tfr}{\mrm{tfr}}
\nc{\Vect}{\mrm{Vect}}
\nc{\sVect}{\mrm{sVect}}
\nc{\fix}{\mrm{fix}}
\nc{\Hilb}{\mathrm{Hilb}}
\nc{\flci}{\mathrm{flci}}
\nc{\Isom}{\mathrm{Isom}}
\nc{\GL}{\mathrm{GL}}
\nc{\BGL}{\mathrm{BGL}}
\nc{\SL}{\mathrm{SL}}
\nc{\Sp}{\mathrm{Sp}}
\nc{\fin}{\mathrm{fin}}
\nc{\cl}{\mathrm{cl}}
\nc{\cn}{\mathrm{cn}}
\nc{\sm}{\mathrm{sm}}
\nc{\heart}{\heartsuit}
\renc{\o}{\mrm{or}}
\nc{\GWpsh}{\mrm{GW}}
\nc{\ev}{\mrm{ev}}
\nc{\Ann}{\mrm{Ann}}
\nc{\FSYN}{\mathcal{FS}\mrm{yn}}
\nc{\FFLAT}{\mathcal{FF}\mrm{lat}}
\nc{\mrk}{\mrm{mrk}}%
\nc{\FFmrk}{\mathcal{FF}\mrm{lat}^{\mrk}}
\nc{\FFnu}{\mathcal{FF}\mrm{lat}^{\mrm{nu}}}
\nc{\FFbas}{\mathcal{FF}\mrm{lat}^{\mrm{bas}}}
\nc{\FQSM}{\mathcal{FQS}\mathrm{m}}
\nc{\Quot}{\mathrm{Quot}}    %% Quot scheme
\nc{\COH}{\mathcal{C}\mathrm{oh}}
\let\phi\varphi

\nc{\Rees}{\mrm{Rees}}

\nc{\robber}{\mathcal{R}}%
\nc{\mv}{\mrm{mv}}
\nc{\const}{\mrm{const}}
\nc{\robbermv}{\robber^{\mv}}
\nc{\robberconst}{\robber^{\const}}
\nc{\robbernot}{\robber_0}%
\nc{\sectionmv}{i^{\mv}}%
\nc{\sectionconst}{i^{\const}}%
\nc{\sectionnot}{i}%
\nc{\st}{\mathrm{st}}

\nc{\FGor}{{\mrm{FGor}}}     %% Finite Gorenstein
\nc{\fgor}{{\mrm{fgor}}}
\nc{\Aug}{{\mrm{Aug}}}  
\nc{\Uni}{{\mrm{Uni}}}  
\nc{\ori}{{\mrm{or}}}        %% Oriented
\nc{\sym}{{\mrm{sym}}}  
\nc{\alt}{{\mrm{alt}}}   
\nc{\nonu}{{\mrm{nu}}}  
\nc{\Ima}{\operatorname{Im}} %% Image of a linear map
\nc{\Hyp}{\operatorname{Hyp}}
\nc{\tel}{\mathrm{tel}}
\nc{\FQSm}{\mathrm{FQSm}}
\nc{\Perf}{\mathrm{Perf}}

\nc{\inftyCat}{\term{$\infty$-category}}
\nc{\inftyCats}{\term{$\infty$-categories}}

\nc{\inftyOneCat}{\term{$(\infty,1)$-category}}
\nc{\inftyOneCats}{\term{$(\infty,1)$-categories}}

\nc{\inftyGrpd}{\term{$\infty$-groupoid}}
\nc{\inftyGrpds}{\term{$\infty$-groupoids}}

\nc{\inftyTop}{\term{$\infty$-topos}}
\nc{\inftyTops}{\term{$\infty$-toposes}}

\nc{\inftyTwoCat}{\term{$(\infty,2)$-category}}
\nc{\inftyTwoCats}{\term{$(\infty,2)$-categories}}

\title{An alternative to spherical Witt vectors}

\author{Thomas Nikolaus}
\address{FB Mathematik und Informatik \\
Universit\"at M\"unster \\
Germany} 
\email{\href{mailto:nikolaus@uni-muenster.de}{nikolaus@uni-muenster.de}}
\urladdr{\url{https://wwwmath.uni-muenster.de/u/nikolaus}}

\author{Maria Yakerson}
\address{CNRS \& IMJ-PRG\\
Paris\\
France}
\email{\href{mailto:yakerson@imj-prg.fr}{yakerson@imj-prg.fr}}
\urladdr{\url{https://www.muramatik.com}}

\date{\today}

\begin{document}

\begin{abstract}
We give a direct construction of the ring spectrum of spherical Witt vectors of a perfect \mbox{$\FF_p$-algebra $R$} as the completion of the spherical monoid algebra $\Ss [R]$ of  the multiplicative monoid $(R,\cdot)$ at the ideal $I = \fib(\Ss[R] \to R)$. This generalizes a construction of Cuntz and Deninger. We also use this to give a description of the category of $p$-complete modules over the spherical Witt vectors and a universal property for spherical Witt vectors as an $\mathbb{E}_1$-ring. 
\end{abstract}

\maketitle

%\vspace{0em}
%\parskip 0.2cm
%
%\parskip 0pt
%%\tableofcontents
%
%\parskip 0.2cm

%\section{Introduction}

The ring of Witt vectors\footnote{Throughout, we refer specifically to the \textit{$p$-typical} Witt vectors.} associated to a commutative ring $R$ is a classical algebraic object. When $R$ is the finite field  $\FF_p$, its ring of Witt vectors is the ring of $p$-adic integers $\Z_p$. More generally, the construction of the ring of  Witt vectors provides a bridge from characteristic $p$ to mixed characteristic $(0,p)$, which makes it an important tool in arithmetic geometry. An element in the ring of Witt vectors $\W(R)$ is given by an infinite sequence of elements in $R$, with addition and multiplication of such sequences being defined not componentwise, but in a non-trivial way, using certain universal polynomials. 
%In particular, the functor of taking Witt vectors gives an equivalence between perfect $\FF_p$-algebras and $p$-torsionfree $p$-complete $\Z$-algebras with perfect $\mod p$ reduction (reference?).

When $R$ is a perfect $\FF_p$-algebra, a slick  construction of the  ring $\W(R)$ was obtained by Cuntz and Deninger in~\cite{CuntzDeningerWitt}.
%,  and later extended to de Rham Witt complexes in~\cite{CuntzDeningerdeRhamWitt}. 
They prove in this case that $\W(R)$ is isomorphic to a completion of the monoid ring $\Z[R]$, with respect to the multiplicative monoid of $R$.
 The completion is taken at the ideal given by the kernel of the canonical ring map $\Z[R] \to R$.  The immediate benefits of this approach are that addition and multiplication are straightforward, and the multiplicative lift\footnote{This is typically called Teichmüller lift, but we will not use this name due to Teichmüller's involvement in the Nazi regime} becomes simply the composition $R \to \Z [R] \to W(R)$.

In higher algebra (a.k.a. brave new algebra) one often replaces the ring of integers with the sphere spectrum $\Ss$, which can be thought of as a deformation of $\Z = \pi_0 \Ss$. %For example, this idea was used in the groundbreaking work of Nikolaus and Scholze~\cite{NikolausScholze} that offered a new perspective on topological cyclic homology, which later lead to lots of new computations in $p$-adic geometry. 
Following this philosophy, Lurie introduced \emph{spherical Witt vectors}, which are a lift of Witt vectors from the ring of $p$-adic integers to the $p$-complete sphere spectrum~\cite{LurieEllipticII}. Spherical Witt vectors are an important ingredient, for example, in the recent work on chromatic homotopy theory by Burklund, Schlank and Yuan~\cite{ChromaticNull}. The definition of spherical Witt vectors is by obstruction theory (see e.g.~\cite[Section 2]{ChromaticNull}) or by the adjoint functor theorem using the universal property (see \cite{Antieau}). We  offer an easy direct construction along the lines of Cuntz and Deninger which morally shows that the spherical Witt vectors are generated by convergent sums of multiplicative lifts:

 \begin{thm}\label{thm:spherical Witt}
Let $R$ be an (ordinary) perfect $\FF_p$-algebra and let $\Ss_{\W(R)}$ be its $\Einfty$-ring spectrum of spherical Witt vectors. Then
there is a canonical equivalence of $\Einfty$-rings induced by the multiplicative lift:
\[
\Ss_{\W(R)} \simeq  \Ss [R]^{\wedge}_I 
\] 
 where the right-hand side is the  completion of the spherical monoid ring $\Ss [R]$ of the  multiplicative monoid of $R$ with respect to the ideal $I = \fib(\Ss [R] \to R)$.
 \end{thm}
 
 Recall that for a map of $\Einfty$-ring spectra $A \to B$ the fibre $I$ is an ideal (in the sense of Smith) and so are 
 $I^n := I \otimes_A \cdots \otimes_A I$. Thus the cofibres $A/I^n$ are also $\Einfty$-rings and we write 
 \[
 A^\wedge_I = \underleftarrow{\lim} \, A / I^n \ .
 \]
Equivalently \cite[Section 2.1]{MNN}, 
%A. Mathew, N. Naumann, and J. Noel. Nilpotence and descent in equivariant stable homotopy theory, Advances in Mathematics 305 (2017), 994–1084.
we have that $A^\wedge_I$ is the limit of the Amitsur complex of $A \to B$
%, which is the cosimplicial diagram
\[ A^{\wedge}_I = \lim_{\Delta}  \left( 
\xymatrix{ 
B\ar[r]<1.5pt>\ar[r]<-1.5pt> & B \otimes_{A} B \ar[r]<3pt>\ar[r]\ar[r]<-3pt> & B \otimes_{A} B \otimes_{A} B \ar[r]<1.5pt>\ar[r]<-1.5pt>\ar[r]<4.5pt> \ar[r]<-4.5pt> & \cdots 
} \right)
\]
 and  is therefore sometimes called the Bousfield--Kan completion. 
 
% \begin{rem}\label{rem_completion}
 Note that in such a situation there is also another notion of completion, namely the Bousfield localization at the ideal $I$. That is the terminal map of ring spectra $A \to A'$ that is a mod $I$ equivalence (i.e. equivalence after $-\otimes_A B$). In general the Bousfield--Kan completion $A \to A^\wedge_I$ is not a mod $I$ equivalence and hence doesn't agree with the Bousfield localization. But if $A \to A^\wedge_I$ is a mod $I$ equivalence then the two notions of completion agree. We will show that this is the case in our situation. 
 
Moreover, all of these completions make sense for arbitrary $A$-modules $M$ and we will show the following result.
% 
% There is also a further notion of completeness referring to elements: to this end let 
% \[
% J_* :=  \ker(\pi_*(A) \to \pi_*(B)) = \mathrm{im}(\pi_*(I) \to \pi_*(A)) .
% \]
%Then we say that an $A$-module $M$ is $J_*$-complete if for every element $x \in J_*$ the inverse limit of 
% \[
%... \xrightarrow{ x\cdot } M[2*] \xrightarrow{ x\cdot } M[*] \xrightarrow{ x\cdot } M
% \]
% vanishes. 
%\end{rem}

\begin{thm}\label{completeness}
The $\infty$-category of $p$-complete modules over $\Ss_{\W(R)}$ embeds by restriction fully faithfully into the $\infty$-category of modules over $\Ss[R]$, i.e. the $\infty$-category of spectra with an  action of the multiplicative monoid $R$.  

For $M$ a bounded below $\Ss[R]$-module we have
\[
 \Ss_{W(R)} \widehat{\otimes}_{\Ss[R]}  M = M^\wedge_I 
\]
where $\widehat{\otimes}$ is the $p$-completed tensor product and  $M^\wedge_I = \underleftarrow{\lim} \, M / I^nM $ with $I^nM := I^n \otimes_{\Ss[R]} M$. Moreover if $M$ is bounded below and $p$-complete then the following are equivalent:
\begin{enumerate}
\item \label{eins}
$M$ lies in the essential image of the embedding of $p$-complete $\Ss_{\W(R)}$-modules;
\item\label{zwei} $M$ is $I$-complete in the sense that 
$M \xrightarrow{\simeq}  M^\wedge_I$ is an equivalence;
\item\label{drei}
$M$ is Bousfield local with respect to the mod $I$ equivalences;
\item \label{vier} The induced multiplicative $R$-action on the $\mathbb{F}_p$-homology groups $H_*(M, \mathbb{F}_p)$ is additive;
\item \label{fuenf}
For every $r,s \in R$ the map 
\[
 \rho_r + \rho_s - \rho_{r + s}: M \to  M
\]
is as a map of spectra divisible by $p$;

\item \label{sechs}
For every $r,s \in R$ the homomorphism 
\[
 \rho_r + \rho_s - \rho_{r + s}: \pi_* M \to \pi_* M
\]
is as a homomorphism of graded abelian groups divisible by $p$;
\item \label{sieben}
The induced multiplicative $R$-action on the abelian groups $\pi_n(M)/p$ and $\pi_n(M)[p]$ is  additive\footnote{The quotient $\pi_n(M)/p$ means the actual, non-derived quotient in abelian groups.}.
\item \label{acht}
The homotopy groups $\pi_n(M)$ of $M$ all lie separately in the essential image of the embedding of $p$-complete $\Ss_{\W(R)}$-modules \footnote{Note that the homotopy groups of $M$ are considered as $\Ss[R]$-modules here. But we can equivalently consider them as  (ordinary) modules over $\Z[R]$ and then the condition is that it is derived $p$-complete and the action of $\Z[R]$ extends to a $W(R)$-action. This in turn then is equivalent to the module being $I$-complete (by condition \ref{zwei}) which again can be expressed in classical terms.};

\end{enumerate}
If $M$ is not bounded below, then \eqref{eins}, \eqref{fuenf}, \eqref{sechs}, \eqref{sieben}, \eqref{acht} are still equivalent and $\eqref{zwei} \Rightarrow \eqref{drei} \Rightarrow \eqref{eins} \Rightarrow \eqref{vier}$.
%\begin{enumerate}\setcounter{enumi}{4}
%\item\label{i-com} $M$ is $I$-complete in the sense that 
%$M \simeq  M^\wedge_I = \underleftarrow{\lim} \, M / I^nM$ 
%where $I^nM := I^n \otimes_{\Ss[R]} M$.
%\item
%$M$ is Bousfield local with respect to the mod $I$ equivalences. 
%\item The induced multiplicative $R$-action on the $\mathbb{F}_p$-homology groups $H_*(M, \mathbb{F}_p)$ is additive.
%%\item 
%%All homotopy groups $\pi_i(M)$ are as modules over $\Z[R]$ complete with respect to $I_0 = \ker(\Z[R] \to R)$.
%\end{enumerate}
%
%A general module $M$ over $\Ss[R]$ lies in the essential image precisely iff all its connective covers $\tau_{\geq n} M$ for 
%$n \to -\infty$ lie in the essential image\footnote{It is enough if this is the case for $\tau_{\geq n_i} M$ with $n_i$ a sequence that tends to $-\infty$.} which is the case precisely iff all homotopy groups %$\pi_n(M)$ considered as $\Ss[R]$-modules 
%lie in the essential image. 
\end{thm}

As a consequence of this description of the category of modules, we also get a universal property of spherical Witt vectors:

\begin{cor}\label{cor_eins}
Let $A$ be a $p$-complete $\mathbb{E}_1$-ring spectrum.  
Then we have a pullback square 
\begin{equation*}%\label{square}
\xymatrix{
\Map_{\mathrm{Alg}_{\mathbb{E}_1}}( \Ss_{\W(R)} , A) \ar[r] \ar[d] & \Map_{\mathrm{Ring}}( R , {\pi_0 A} / {p} ) \ar[d]\\
\Map_{\mathrm{Mon}_{\mathbb{E}_1}}\Big( (R,\cdot) , (\Omega^\infty A, \cdot)\Big) \ar[r] & \Map_{\mathrm{Mon}}\Big( (R,\cdot) ,  ({\pi_0A} / {p},\cdot ) \Big)
}
\end{equation*}
Here $\pi_0A / p$ is the ordinary quotient of $\pi_0A$ by the ideal $(p)$, which is automatically a two sided ideal, hence $\pi_0A / p$ is an ordinary associative ring.% which for $n \geq 2$ is a commutative ring. 
\end{cor}

In other words: $\Map_{\mathrm{Alg}_{\mathbb{E}_1}}( \Ss_{\W(R)} , A)$ is the full subspace of $\Map_{\mathrm{Mon}_{\mathbb{E}_1}}\Big( (R,\cdot) , (\Omega^\infty A, \cdot)\Big)$ consisting of those multiplicative maps for which the composition $R \to \pi_0A / p$ becomes a ring map. 
Another reformulation of Corollary \ref{cor_eins}  is that the spherical Witt vectors $\Ss_{\W(R)}$ are the universal $p$-complete $\mathbb{E}_1$-ring with a multiplicative map from $R$ which becomes additive on $\pi_0(\Ss_{\W(R)}) / p$. 

\begin{remark}\label{remark_en}
With  the same proof as Corollary \ref{cor_eins} one can deduce a similar statement for a $p$-complete $\mathbb{E}_n$-ring spectrum $A$ and $\mathbb{E}_n$-maps out of spherical Witt vectors for $n \geq 1$, namely that we have a pullback:
\begin{equation}\label{square}
\xymatrix{
\Map_{\mathrm{Alg}_{\mathbb{E}_n}}( \Ss_{\W(R)} , A) \ar[r] \ar[d] & \Map_{\mathrm{Ring}}( R , {\pi_0 A} / {p} ) \ar[d]\\
\Map_{\mathrm{Mon}_{\mathbb{E}_n}}\Big( (R,\cdot) , (\Omega^\infty A, \cdot)\Big) \ar[r] & \Map_{\mathrm{Mon}}\Big( (R,\cdot) ,  ({\pi_0A} / {p},\cdot ) \Big) \ .
}
\end{equation}
The reason we did not state this here is that when $n>1$ one can prove that both horizontal maps are in fact equivalences (see Proposition \ref{cor_2} below) . For the case $n = \infty$ this reproves the universal property of spherical Witt vectors given in  \cite{LurieEllipticII, ChromaticNull} and used in \cite{Antieau} to construct the spherical Witt vectors.

For the case $n = \infty$ it is easy to see that the lower horizontal map in \eqref{square} is an equivalence: by perfectness of $R$ we may replace $(\Omega^\infty A, \cdot)$ by the inverse limit of the $p$-th power map on $\Omega^\infty A$ which respects the multiplicative $\mathbb{E}_\infty$-structure. Using that the $p$-th power map vanishes on higher homotopy groups, we can see that this inverse limit is equivalent to the inverse limit perfection of $({\pi_0A} / {p},\cdot )$, i.e. the tilt of $\pi_0(A)$. Then it follows by the fact that \eqref{square} is a pullback that the upper horizontal map is also an equivalence. This then gives a new proof of the aforementioned universal property of $\Ss_{\W(R)}$ as an $\mathbb{E}_\infty$-ring. %In fact, one can do a bit better, namely with the same obstruction theory that is used to proof the $\mathbb{E}_\infty$-universal property of $\Ss_{\W(R)}$ one can prove a slightly stronger universal property. While this statement is logically independent of the results of this paper we include it for completeness as the next proposition. 
\end{remark}

For the next statement we use the same obstruction theory that was initiated by Lurie to prove the $\mathbb{E}_\infty$-universal property of spherical Witt vectors. 
While the proof of this statement is logically independent of the results of this paper we include it for completeness. But note that in the case $n = \infty$ we did sketch a direct proof using the results of this paper in the previous Remark \ref{remark_en}.  We thank Maxime Ramzi for a discussion of the next Statement and allowing us to include his proof in this paper and especially his example that shows that the condition that $\pi_0(A)$ lies in the center of $\pi_*(A)$ is necessary, see Remark  \ref{example_maxime}. 

%However in this case one can prove  that the lower horizontal map in the square of Corollary \ref{cor_eins} is an equivalence (see Lemma \ref{lem_discrete} below for the $\mathbb{E}_1$-case). Thus we get an even stronger universal property generalizing the universal property of Witt vectors from \cite{Lurie, BurklungSchlankYuan}:

\begin{prop}\label{cor_2}
Let $A$ be a $p$-complete $\mathbb{E}_n$-ring spectrum for $1 \leq n \leq \infty$. For $n = 1$ assume additionally that $\pi_0(A)$ lies in the center of $\pi_*(A)$. Then we have a natural equivalence
\[
\Map_{\mathrm{Alg}_{\mathbb{E}_n}}( \Ss_{\W(R)} , A) \simeq  \Map_{\mathrm{CRing}}( R , {\pi_0 A}^\flat)
\]
where the tilt $ {\pi_0 A}^\flat$ is the inverse limit perfection of $\pi_0(A)/p$. 
\end{prop}

In particular this implies that every $\mathbb{E}_1$-map uniquely extends to an $\mathbb{E}_n$-map. It also follows that for $n > 2$ or under the additional assumption for $n=1$ both horizontal maps in the square \eqref{square} of Remark \ref{remark_en} are equivalences. For the upper line this is the statement of Proposition \ref{cor_2} and for the lower one it is the Statement of Proposition \ref{cor_2} for $R$ replaced by $\mathbb{F}_p[R]$ using that $\mathbb{S}[R]^\wedge_p = \Ss_{\W(\mathbb{F}_p[R])}$.  \\
%
%can do better and get an anlogu
%
%Let us also note that the universal property of Witt vectors from \cite{Lurie, BurklungSchlankYuan}, that has been used to construct them in \cite{Benspaper}, easily follows from our description here. Namely for $n = \infty$ the lower horizontal map in \eqref{square} is an equivalence by perfectness of $R$ and thus the upper horizontal map becomes an equivalence as well. 

Finally we note that all of our constructions and statements have, with exactly the same proofs, analogues for ordinary Witt vectors: for every perfect $\FF_p$-algebra $R$ we have an equivalence
\[
W(R) \simeq \underleftarrow{\lim} (\Z[R] / I^n)
\]
which is a derived version of Cuntz--Deninger's statement. We also have that the full subcategory of $p$-complete objects in the derived category of $W(R)$ is equivalent to a full subcategory of $\Z[R]$-modules characterized by the analogous conditions to $\eqref{zwei} - \eqref{sechs}$ of Theorem \ref{completeness}, with the difference that we do not have to assume bounded below anywhere since $\mathbb{F}_p$ is dualisable over $\Z$ (but not over $\Ss$ where it is only pseudocoherent). We then have the exact analogues of  Corollary \ref{cor_eins} and Proposition \ref{cor_2} for $W(R)$. 

\ssec*{Acknowledgments}
We would like to thank Ben Antieau, Joachim Cuntz, Christopher Deninger, Felix Gu, Achim Krause, Jonas McCandless, Maxime Ramzi, and Tomer Schlank for useful discussions. 

We also thank Department of Mathematics at M\"unster University for hospitality during the visit of the second author, when the work on this project was conducted.
The first author was funded by the Deutsche Forschungsgemeinschaft (DFG, German Research Foundation) -- Project-ID 427320536--SFB 1442, as well as under Germany's Excellence Strategy EXC 2044 390685587, Mathematics M\"unster: Dynamics-Geometry-Structure. 
The second author is grateful to CNRS and Institut de Mathématiques de Jussieu --- Paris Rive Gauche for perfect working conditions.
%
% In fact, all of our statements also have analogues for ordinary Witt vectors which are also partly new and which we state for completeness. TODO

\section{Proofs of the Results}

%We continue to let $R$ be a perfect $\FF_p$-algebra. 
%Consider the Teichm\"uller lift $t \colon R \to \W(R)$. 
First observe that we have  a map $t \colon \FF_p[R] \to R$ from the $\FF_p$-monoid algebra on the  multiplicative monoid of $R$ to $R$. We note that $\FF_p[R]$ is also a perfect $\FF_p$-algebra and therefore we get an induced map 
%By obstruction theory, it uniquely lifts to a map 
\[
T \colon \Ss_{\W(\FF_p[R])} \to \Ss_{\W(R)} 
\]
using functoriality of the spherical Witt vectors.
% (cite Lurie's thm that gives a bijection between perfect Fp-algebras and their spherical Witt vectors?). 
The $\Einfty$-ring of spherical Witt vectors  $\Ss_{\W(\FF_p[R])}$ is the unique $p$-complete lift of $\FF_p[R]$ to the sphere. Since $\Ss [R]^{\wedge}_p$ is a $p$-complete $\Einfty$-ring whose $\FF_p$-homology is $\FF_p[R]$
 we deduce that  $\Ss_{\W(\FF_p[R])} \simeq \Ss [R]^{\wedge}_p$.
Thus we can interpret $T$ as a map 
\[
T \colon \Ss [R]^{\wedge}_p \to \Ss_{\W(R)}.
\]
We have a factorization 
$
 \Ss [R]^{\wedge}_p \to \Ss_{\W(R)} \to R
$
of the canonical map $\Ss [R]^{\wedge}_p \to R$ through $T$ which follows by base-changing the first two rings to $\FF_p$ (since $R$ is an $\FF_p$-algebra). \footnote{Note that on $\pi_0$ the map $T$ induces a map $\Z[R]^\wedge_p \to \W(R)$ which corresponds to the usual multiplicative character $R \to W(R)$. This follows again by the same obstruction theory: $\Z[R]^\wedge_p \simeq \W(\FF_p[R])$ and there is only one lift of $\FF_p[R] \to R$ to a map $ \W(\FF_p[R]) \to \W(R)$. }

Recall that a map of ring spectra $A \to B$ is called \textit{$p$-complete homological epimorphism} if the map $B \widehat{\otimes}_A B \to B$ is an equivalence, where $\widehat{\otimes}$ means the $p$-completed tensor product. 
We have the following well-known properties:
\begin{lem}\label{prop_epi}
\begin{enumerate}
\item A map of connective ring spectra $A \to B$ is a $p$-complete homological epimorphism precisely if $A \otimes_\Ss \FF_p \to B \otimes_\Ss \FF_p$ is a homological epimorphism.
\item Any surjective map of ordinary perfect $\FF_p$-algebras is a homological epimorphism.
\item For a $p$-complete homological epimorphism $A \to B$ the restriction functor from $p$-complete $B$-modules to $A$ modules is fully faithful. 
\item If $A \to B$ is a $p$-complete homological epimorphism and $M, N$ are $B$-modules, then the canonical map
\[
M \widehat{\otimes}_A N \to M \widehat{\otimes}_B N
\]
is an equivalence.
\end{enumerate}
\end{lem}
\begin{proof}
For (1) we simply note that $
(B \otimes_\Ss \FF_p) \otimes_{A \otimes_\Ss \FF_p} (B \otimes_\Ss \FF_p) \simeq (B \widehat{\otimes}_A B) \otimes_\Ss \FF_p
$
and that we can check equivalences between $p$-complete, connective spectra after base-change to $\FF_p$. 

For (2) note that $B \otimes^{L}_A B$ is an animated perfect $\FF_p$-algebra and therefore discrete, but from surjectivity of $A \to B$ it easily follows that $\pi_0$ is isomorphic to $B$.

For (3) we simply note that the restriction functor has a left adjoint $B \widehat{\otimes}_A -: \Mod(A)^\wedge_p \to  \Mod(B)^\wedge_p$ and thus is fully faithful precisely if for any $B$-module $M$ the counit $ B\widehat{\otimes}_A M \to M$ is an equivalence. Since both sides commute with colimits in $M$ it suffices to check this for $M = B$. 

For (4) we note that both sides commute with colimits in $M$ and $N$ and thus we can reduce to $M = N = B$.
\end{proof}

\begin{lem}\label{epi}
\label{lm: hom epi}
The map $T \colon \Ss [R]^{\wedge}_p \to \Ss_{\W(R)}$ is a $p$-complete homological epimorphism.
\end{lem}

\begin{proof}
By (1) of the previous Lemma it suffices to check that $\FF_p[R] \to R$ is a homological epimorphism. But this is a surjective map of perfect $\FF_p$-algebras, so a homological epimorphism by (2).
\end{proof}

\begin{proof}[Proof of Theorem \ref{thm:spherical Witt}]
The completion $\Ss [R]^{\wedge}_I$ which we want to understand is the inverse limit of the corresponding Amitsur complex:
\[ \Ss [R]^{\wedge}_I = \lim_{\Delta}  \left( 
\xymatrix{ 
R\ar[r]<1.5pt>\ar[r]<-1.5pt> & R \otimes_{\Ss [R]} R \ar[r]<3pt>\ar[r]\ar[r]<-3pt> & R \otimes_{\Ss [R]} R \otimes_{\Ss [R]} R \ar[r]<1.5pt>\ar[r]<-1.5pt>\ar[r]<4.5pt> \ar[r]<-4.5pt> & \cdots 
} \right)
\]
Since $R$ is an $\FF_p$-algebra we may replace the tensor products by completed tensor products and then can replace the tensor products using (3) of Lemma \ref{prop_epi}  by tensor products over $\Ss_{\W(R)}$:
\[ \Ss [R]^{\wedge}_I = \lim_{\Delta}  \left( 
\xymatrix{ 
R\ar[r]<1.5pt>\ar[r]<-1.5pt> & R \otimes_{\Ss_{\W(R)}} R \ar[r]<3pt>\ar[r]\ar[r]<-3pt> & R \otimes_{\Ss_{\W(R)}} R \otimes_{\Ss_{\W(R)}} R \ar[r]<1.5pt>\ar[r]<-1.5pt>\ar[r]<4.5pt> \ar[r]<-4.5pt> & \cdots 
} \right)
\]
The cosimplicial diagram on the right-hand side is the Amitsur complex of the map $\Ss_{\W(R)} \to R$.
Using that $R = \FF_p  \otimes_{\Ss} \Ss_{\W(R)} $ we find that this cosimplicial diagram is also
given by the base-change of the Amitsur complex of $\Ss \to \FF_p$ along $\Ss \to \Ss_{\W(R)}$, i.e.
\[
\left( 
\xymatrix{ 
 \FF_p\otimes_\Ss \Ss_{\W(R)}  \ar[r]<1.5pt>\ar[r]<-1.5pt> &   \FF_p \otimes_\Ss \FF_p \otimes_\Ss \Ss_{\W(R)}\ar[r]<3pt>\ar[r]\ar[r]<-3pt> &   \FF_p \otimes_\Ss \FF_p \otimes_{\Ss} \FF_p  \otimes_\Ss\Ss_{\W(R)}\ar[r]<1.5pt>\ar[r]<-1.5pt>\ar[r]<4.5pt> \ar[r]<-4.5pt> & \cdots 
} \right)
\]
This is the Adams tower for  $\Ss_{\W(R)}$ 
which converges to  $\Ss_{\W(R)}$ since  $\Ss_{\W(R)}$ is connective and $p$-complete, see e.g. \cite[Theorem 6.6]{Bousfield}. This  completes the proof of the Theorem.
\end{proof}

\begin{remark}
The following observation is due to Tomer Schlank: Instead of considering the multiplicative monoid \((R, \cdot)\), one can take any \( p \)-divisible monoid \( M \) equipped with a surjective map $\mathbb{F}_p[M] \to R$ of rings. Under this assumption, the exact same proof shows that the conclusion of Theorem \ref{thm:spherical Witt} remains valid, yielding an equivalence  
\[
\Ss_{\W(R)} \simeq \Ss [M]^{\wedge}_I
\]  
where $I = \fib(\Ss [M] \to R)$. Similarly, other results in this paper continue to hold in this generality, with the necessary modifications.  
\end{remark}

\begin{proof}[Proof of Theorem \ref{completeness}]
By Lemma \ref{epi} and Lemma \ref{prop_epi} we find that the restriction $\Mod(\Ss_{\W(R)})^\wedge_p \to \Mod(\Ss[R])$ is fully faithful. The essential image is exactly given by those $p$-complete modules where the action extends to an $\Ss_{W(R)}$-action, in which case the extension is unique. 
Note that the left adjoint to restriction is given by 
\[
M \mapsto \Ss_{W(R)} \widehat{\otimes}_{\Ss[R]}  M \ .
\]
We want to show that for $M$ bounded below this agrees with the Bousfield--Kan completion
\[
M \to M^\wedge_I = \underleftarrow{\lim} \, M / I^nM \ . 
\]
To see this we compute $M^\wedge_I$ by the Amitsur complex
\[
M^{\wedge}_I = 
\lim_{\Delta}  \left( 
\xymatrix{ 
R \otimes_{\Ss [R]} M\ar[r]<1.5pt>\ar[r]<-1.5pt> & R \otimes_{\Ss [R]}R \otimes_{\Ss [R]} M \ar[r]<3pt>\ar[r]\ar[r]<-3pt> & R \otimes_{\Ss [R]}R \otimes_{\Ss [R]} R \otimes_{\Ss [R]} M \ar[r]<1.5pt>\ar[r]<-1.5pt>\ar[r]<4.5pt> \ar[r]<-4.5pt> & \cdots 
} \right) \ .
\]
Now we proceed similar to the proof of Theorem \ref{thm:spherical Witt}. The cosimplicial diagram is  equivalent to
\[
\lim_{\Delta}  
\left( 
\xymatrix{ 
R  \widehat{\otimes}_{\Ss [R]} M\ar[r]<1.5pt>\ar[r]<-1.5pt> & R \widehat{\otimes}_{\Ss_{W(R)}}R  \widehat{\otimes}_{\Ss [R]} M \ar[r]<3pt>\ar[r]\ar[r]<-3pt> & R  \widehat{\otimes}_{\Ss_{W(R)}}R  \widehat{\otimes}_{\Ss_{W(R)}} R  \widehat{\otimes}_{\Ss [R]} M \ar[r]<1.5pt>\ar[r]<-1.5pt>\ar[r]<4.5pt> \ar[r]<-4.5pt> & \cdots 
} \right)
\]
And thus by using that $R = \FF_p  \otimes_{\Ss} \Ss_{\W(R)} $ we see that this is given by the Adams tower of  $\Ss_{\W(R)}  \widehat{\otimes}_{\Ss [R]} M$:
\[
\left( 
\xymatrix{ 
 \FF_p\otimes_\Ss \Ss_{\W(R)}  \widehat{\otimes}_{\Ss [R]} M \ar[r]<1.5pt>\ar[r]<-1.5pt> &   \FF_p \otimes_\Ss \FF_p \otimes_\Ss \Ss_{\W(R)} \widehat{\otimes}_{\Ss [R]} M\ar[r]<3pt>\ar[r]\ar[r]<-3pt> %&   \FF_p \otimes_\Ss \FF_p \otimes_{\Ss} \FF_p  \otimes_\Ss\Ss_{\W(R)}\otimes_{\Ss [R]} M\ar[r]<1.5pt>\ar[r]<-1.5pt>\ar[r]<4.5pt> \ar[r]<-4.5pt> 
 & \cdots 
} \right)
\]
which converges to $\Ss_{\W(R)}  \widehat{\otimes}_{\Ss [R]} M$ by boundedness. 

Now we let $M$ be arbitrary, that is not necessarily bounded below and show the implications that hold in this generality: \\
$\eqref{eins} \Rightarrow \eqref{fuenf}$: we have a factorization $\Ss[R] \to \Ss_{\W(R)} \to \mathrm{map}(M,M)$ as maps of spectra. Taking the mod $p$-reduction we get a factorization 
$\Ss[R] \to \Ss_{\W(R)}/p \to \mathrm{map}(M,M)/p$. We have that $\pi_0(\Ss_{\W(R)}/p) = R$ and the first map identifies on $\pi_0$ the elements $[r] + [s]$ and $[r+s]$ so that the claim follows.

$\eqref{fuenf} \Rightarrow \eqref{sechs} \Rightarrow \eqref{sieben}$ are clear, the latter since the homotopy groups of the derived mod $p$ reduction $\pi_n(M)/\!/p$ are given by $\pi_n(M)/p$ and $\pi_n(M)[p]$.

$\eqref{sieben} \Rightarrow \eqref{acht}$: by assumption we have that $\pi_n(M)/p$ and $\pi_n(M)[p]$ are in the essential image of the restriction. 
The image of the restriction functor is closed under limits, extensions and shifts. Thus also the derived mod $p$-reduction $\pi_n M / \! / p$ lies in the essential image. By taking iterated extensions, so do $\pi_n M / \! / p^k$. Since $\pi_n M$ is derived $p$-complete it is the inverse limit of $\pi_n M / \! / p^k$ over $k$ and therefore also in the essential image. 

$\eqref{acht} \Rightarrow \eqref{eins}$: Using the Postnikov tower of $M$ we see that it is a colimit-limit of extensions of $\pi_nM$ and thus the claim follows since the essential image of restriction is  closed under (co)limits, extensions and shifts.

$\eqref{zwei} \Rightarrow \eqref{drei}$: By assumption $M$ is an inverse limit of extensions of $R$-modules (considered as $\Ss[R]$-modules by restriction).
But $R$-modules are clearly mod $I$ local and therefore so is $M$.

$\eqref{drei} \Rightarrow \eqref{eins}$:  Since $\Ss[R] \to \Ss_{\W(R)}$ is a $p$-complete homological epimorphism we find that the restricted $p$-complete modules are the local objects for maps of $\Ss[R]$-modules which are equivalences after base-change to $\Ss_{\W(R)}$. Clearly these maps are still equivalences after base-change to $R$ since we have a factorization $\Ss[R] \to \Ss_{\W(R)} \to R$. Thus the mod $I$ local objects are also restricted along $\Ss[R] \to \Ss_{\W(R)}$.

$\eqref{eins} \Rightarrow \eqref{vier}$:  if $M$ is restricted along $\Ss[R]  \to \Ss_{\W(R)}$ then the base-change to $\mathbb{F}_p$ becomes a module over 
$\Ss_{\W(R)} \otimes_\Ss \mathbb{F}_p = R$, so that on the homology we have an additive $R$-action. 

Finally assume that $M$ is bounded below, then we show $\eqref{vier} \Rightarrow \eqref{eins}$: note that $R$-modules are certainly in the essential image and thus also the homology $M \otimes \mathbb{F}_p$ by assumption \eqref{vier} (using also condition \eqref{acht}) and also iterated homologies, i.e. $M \otimes \mathbb{F}_p \otimes ... \otimes \mathbb{F}_p$. Then by the convergence of the Bousfield Kan tower this implies that $M$ is a limit of restricted modules, hence restricted. 
%To show $\eqref{drei} \Leftrightarrow \eqref{eins}$ it suffices to show that for every $\Ss[R]$-module $M$ the map $M \to \Ss_{W(R)} \widehat{\otimes}_{\Ss[R]}  M$ is an equivalence 
%after base-change  along $\Ss[R] \to R$ since the target is certainly Bousfield local (by the fact that it is also the limit of the Bousfield--Kan tower as shown in the previous part of the proof). This base-change is given by the map
%\[
%R\otimes_{\Ss[R]} M  \to R \otimes_{\Ss[R]} \Ss_{W(R)} \widehat{\otimes}_{\Ss[R]}  M \ . %=  R \widehat{\otimes}_{\Ss[R]} M^\wedge_I =  R \widehat{\otimes}_{\Ss_{W(R)}} M^\wedge_I
%\]
%%and is $p$-complete. Both sides commute with colimits in $M$, the right-hand side since we have already identified the completion with a base change. Thus we can reduce to the case $M = \Ss[R]$. In this case the map is given by the equivalence 
%%\[
%%R =  \Ss[R]  \otimes_{\Ss[R]} R \to \Ss_{\W(R)} \otimes_{\Ss[R]} R = \Ss_{\W(R)} \widehat{\otimes}_{\Ss[R]} R = \Ss_{\W(R)} \widehat{\otimes}_{\Ss_{\W(R)}} R = R
%%\]
%Using that $R$ is $p$-torsion and then that $\Ss[R]  \to \Ss_{\W(R)} $ is a $p$-complete homological epimorphism the target simplifies to $R \widehat{\otimes}_{\Ss_{W(R)}} \Ss_{W(R)} \widehat{\otimes}_{\Ss[R]}  M = R\otimes_{\Ss[R]} M$ and the map is an equivalence.
\end{proof} 

\begin{proof}[Proof of Corollary \ref{cor_eins}]
By \cite[Corollary~3.4.1.7]{HA} 
the slice category of $\mathbb{E}_1$-algebras under $\Ss_{\W(R)}$ is equivalent to the $\infty$-category of $\mathbb{E}_1$-algebras in the monoidal $\infty$-category ${}_{\Ss_{\W(R)}} \Mod_{\Ss_{\W(R)}}$ of $\Ss_{\W(R)}$-bimodules equipped with the bimodule tensor product. Since a bimodule is simply a right module in left modules we deduce from Theorem \ref{completeness} that this category embeds by restriction (operadically) fully faithfully into the $\infty$-category of bimodules ${}_{\Ss[R]} \Mod_{\Ss[R]}$ and that the image consists of those $p$-complete modules $M$ such that the actions from both sides induced $R$-actions on $\pi_nM / p$ as well as $\pi_nM[p]$. For an algebra $A$ in ${}_{\Ss[R]} \Mod_{\Ss[R]}$ the multiplicative action of $R$ on $\pi_*A/p$ and $\pi_*A[p]$ from both sides factors over $\Z[R] \to \pi_0 A \to \pi_0A/p$. Thus it lies in the image from both sides precisely if the map $\Z[R] \to \pi_0 A \to \pi_0A/p$ factors over $R$. In summary we have proven that the category of $p$-complete  $\mathbb{E}_1$-algebras under $\Ss_{\W(R)}$ is equivalent to the category of $p$-complete  $\mathbb{E}_1$-algebras $A$ under $\Ss[R]$ with the property that the map $\Z[R] = \pi_0\Ss[R] \to \pi_0A /p$ factors over $R$. From this the claim follows immediately.  
\qedhere
\end{proof}

We thank Maxime Ramzi for explaining the proof of Proposition~\ref{cor_2} which we present here and for explaining the subsequent example to us. In the proof we will use repeatedly the following lemma.

\begin{lem}\label{square-zero}
Let $A$ be a $p$-complete $\mathbb{E}_n$-ring spectrum for $n \geqslant 2$ and $\tilde{A} \to A$ a square zero extension by a $p$-complete symmetric bimodule $P$. Then any $\mathbb{E}_1$-ring map $\Ss_{\W(R)} \to A$ has a unique lift to $\tilde{A}$ up to equivalence.
\end{lem}

\begin{proof}
A lift to $\tilde{A}$ is the same data as a nullhomotopy of the corresponding map of bimodules (a derivation) 
$L_{\Ss_{\W(R)} / \Ss} \to \Sigma P$ where 
\[
L_{\Ss_{\W(R)} / \Ss} = \fib( \Ss_{\W(R)}  \otimes  \Ss_{\W(R)}  \to \Ss_{\W(R)})
\]
 is the $\mathbb{E}_1$-cotangent complex. Since $P$ is symmetric, such a bimodule map is equivalently a module map from the symmetrization of $L_{\Ss_{\W(R)} / \Ss}$ which is the fibre of the canonical map $\alpha \colon \Ss_{\W(R)} \to \mathrm{THH}(\Ss_{\W(R)})$. The map $\alpha$ is an $\FF_p$-equivalence  because $\mathrm{HH}_*(R; {\mathbb{F}_p}) \simeq R$ for a perfect ${\mathbb{F}_p}$-algebra $R$. Since $\alpha$ is a map between connective spectra, it is then also a $p$-adic equivalence, i.e.,  the fibre of $\alpha$ vanishes $p$-adically. Hence the space of nullhomotopies under consideration is trivial, so the space of lifts in the claim is contractible, and the claim follows. 
\end{proof}

\begin{proof}[Proof of Proposition \ref{cor_2}]
We may without loss of generality assume that $A$ is connective. We first prove the case $n = 1$.
By perfectness of $R$ we have that 
\[
\Map_{\mathrm{CRing}}( R , {\pi_0 A}^\flat) = 
\Map_{\mathrm{CRing}}( R , \pi_0(A)/p) = 
\Map_{\mathrm{Alg}_{\mathbb{E}_n}}( \Ss_{\W(R)} , \pi_0(A)/p) \ .
\]
Thus it suffices to shows that the natural map 
\begin{equation}\label{mapstoeq}
\Map_{\mathrm{Alg}_{\mathbb{E}_n}}( \Ss_{\W(R)} , A) \to \Map_{\mathrm{Alg}_{\mathbb{E}_n}}( \Ss_{\W(R)} , \pi_0(A)/p) 
\end{equation}
induced by $A \to \pi_0(A)/p$ is an equivalence. We will decompose it as \[A \to \pi_0(A) \to \pi_0(A)//p \to \pi_0(A)/p\] and show that at each step a lift of a map from $\Ss_{\W(R)}$ is unique up to equivalence. 

First, we consider the space of lifts of a map $\Ss_{\W(R)} \to \pi_0(A)$ to $A$. 
We write $A$ as the  limit of its Postnikov tower $A = \lim \tau_{\leq m} A$. It suffices to show that the lift is unique up to equivalence for each $\tau_{\leq m}A$. Note that each $\tau_{\leq m+1} A$ is a square zero extension of $\tau_{\leq m} A$ by $\Sigma^{m+1}(\pi_{m+1}A)$. 
The assumption that $\pi_0(A)$ lies in the center of $\pi_*(A)$ ensures that the bimodule $\pi_{m+1}A$ is symmetric, in the sense that it is induced by a left module (using that $\pi_0A$ is commutative). Thus we can apply Lemma~\ref{square-zero} to deduce the uniqueness of each lift.

Next, we consider the space of lifts of a map $\Ss_{\W(R)} \to \pi_0(A) // p$ to $\pi_0(A)$. 
Here we argue similarly: we write $\pi_0(A) = \lim \pi_0(A) /\!/ p^m$ by $p$-completeness, then observe that each $\pi_0(A) /\!/ p^{m+1}$ is a square zero extension of $\pi_0(A) /\!/ p^{m}$ by $\pi_0(A) // p$ and apply Lemma~\ref{square-zero}. 

Finally, we use that $\pi_0(A) // p$ is a square zero extension of $\pi_0(A) / p$ by $\Sigma \pi_0 A [p]$ and apply Lemma~\ref{square-zero} again.
This finishes the proof for $n = 1$. 

For $n > 1$ we argue inductively: using Dunn additivity the space of $\mathbb{E}_{n}$-maps $\Map_{\mathrm{Alg}_{\mathbb{E}_n}}( \Ss_{\W(R)} , A)$ can be written as the totalization of a cosimplicial diagram 
\[
[k] \mapsto \Map_{\mathrm{Alg}_{\mathbb{E}_{n-1}}}( (\Ss_{\W(R)})^{\hat{\otimes}_\Ss (k+1)} , A) \ .
\]
We have that $(\Ss_{\W(R)})^{\hat{\otimes}_{\Ss} (k+1)} = \Ss_{\W\left(R^{\otimes_{\mathbb{F}_p} (k+1) }\right)}$ so that by the inductive hypothesis this cosimplicial diagram is equivalent to 
\[
[k] \mapsto \Map_{\mathrm{Ring}}( R^{\otimes_{\mathbb{F}_p} (k+1)}  , {\pi_0 A}^\flat ) \ 
\]
which agrees with the space of $\mathbb{E}_n$-maps $R \to {\pi_0 A}^\flat$, which are the same as commutative ring maps by discreteness. This shows the case of finite $n$ and the case $n = \infty$ follows since it is the limit of all the spaces of $\mathbb{E}_n$-maps, hence the limit of a constant diagram. 
%Since the $p$-th power on $M$ is an equivalence we may replace the $\mathrm{E}_\infty$-monoids $(\Omega^\infty A, \cdot)$ and $\sfrac{\pi_0A} {(p)}$ by their `multiplicative inverse limit perfection', i.e. the inverse limit of the $p$-power maps. We get a commutative square
%\[
%\xymatrix{
%(\Omega^\infty A, \cdot)^{\mathrm{perf}} \ar[r] \ar[d]&  (\Omega^\infty A, \cdot) \ar[d]  \\
%(\sfrac{\pi_0A} {(p)})^{\mathrm{perf}} \ar[r] & \sfrac{\pi_0A} {(p)}
%}
%\]
%and the assertion of the lemma follows once we know the left vertical map is an equivalence. For this we use that we can factor this map as $(\Omega^\infty A)^{\mathrm{perf}} \to \pi_0(A)^{\mathrm{perf}} \to (\sfrac{\pi_0A} {(p)})^{\mathrm{perf}}$. The fact that the second map is an equivalence is \cite{} and for the first map we use that ...
\end{proof}

\begin{remark}[Ramzi]\label{example_maxime}
From the proof of Proposition~\ref{cor_2} we can see that the assumption on $\pi_0(A)$ being in the center is necessary. Otherwise, let $A$ be a square zero extension of $\Ss_{\W(R)}$ by the (non-symmetric) bimodule $L_{\Ss_{\W(R)} / \Ss}$. If the claim of Proposition~\ref{cor_2} would hold for such $A$, it would imply that the space of bimodule maps $L_{\Ss_{\W(R)} / \Ss} \to \Sigma L_{\Ss_{\W(R)} / \Ss}$ is contractible. In particular, its $\pi_1$ would be trivial, so $\Sigma L_{\Ss_{\W(R)} / \Ss} = 0$. By definition of the $\mathbb{E}_1$-cotangent complex, this would imply that  $\Ss_{\W(R)}$ is idempotent. Hence $R = \Ss_{\W(R)} \otimes_\Ss \FF_p$ would also be  idempotent, which is not true unless $R = \FF_p$.

\end{remark}

\bibliographystyle{alphamod}

\let\mathbb=\mathbf

{\small
\bibliography{references}
}

\parskip 0pt

\end{document}